\documentclass[a4paper,12pt]{article}
\usepackage{amsthm}
\usepackage{amsmath}
\usepackage{amsfonts}
\usepackage{amssymb}
\title{Homomorphisms between fundamental groups of K\"ahler manifolds}
\author{Botong Wang}
\begin{document}
\maketitle
\begin{abstract}
We construct a group homomorphism, which can be realized as the induced homomorphism of fundamental groups from a holomorphic map between compact K\"ahler manifolds, but can not be realized by a holomorphic map between smooth projective varieties. And it is also proved that there exists no such example between abelian groups.
\end{abstract}
\section{Introduction}

\newtheorem{thm}{Theorem}[section]
\newtheorem{lemma}[thm]{Lemma}

The fundamental groups of compact K\"ahler manifolds and smooth complex projective varieties have been studied for more than twenty years. Arapura had a survey paper \cite{Arapura} on this subject, and a more detailed introduction was given in the book \cite{ABC}. A natural and unsolved question is whether the class of fundamental groups of compact K\"ahler manifolds and the class of fundamental groups of smooth complex projective varieties coincide. In \cite{Voisin}, Voisin gave examples of compact K\"ahler manifolds of dimension four or higher, which do not have the homotopy type of any smooth complex projective variety. In this paper, we use Voison's method to answer a weaker question. If one considers not only the fundamental groups, but also the homomorphisms between them, we can give a negative answer. 

First, we need some definitions. A group is called \textit{K\"ahler} (or \textit{projective}) if it is isomorphic to the fundamental group of a compact K\"ahler manifold (or smooth complex projective variety). A homomorphism between K\"ahler (or projective) groups is called \textit{K\"ahler} (or \textit{projective}) if it is the homomorphism of fundamental groups induced by a holomorphic map between compact K\"ahler manifolds (or smooth complex projective varieties). The main result of this paper is the following

\begin{thm}
There exists a K\"ahler homomorphism $f: A\to B$ which is not projective. However, if $f$ is K\"ahler, $A$ and $B$ are both abelian groups, then $f$ must be projective.
\end{thm}

I thank my advisor Prof. Donu Arapura who asked me this question and gave me many helpful suggestions.
\section{Between abelian groups}
In this section, we prove that for a homomorphism in the category of abelian groups, the condition of being K\"ahler is same as being projective.

Let me start with some basic properties about K\"ahler groups and K\"ahler homomorphisms.
\begin{lemma}\label{even}
The abelianization of any K\"ahler group has even rank.
\end{lemma}
\begin{proof}
By Hodge theory, the first Betti number of a compact K\"ahler manifold must be even. And for any manifold $X$, 
$$H_1(X,\mathbb{C})\cong \frac{\pi_1(X)}{[\pi_1(X),\pi_1(X)]}\otimes \mathbb{C}$$
\end{proof}
\begin{lemma}
The kernel, cokernel and image of a K\"ahler homomorphism must have even ranks. More precisely, suppose $f:A\to B$ is K\"ahler, and
$$\bar{f}:\frac{A}{[A,A]}\to \frac{B}{[B,B]}$$
is the abelianization of $f$, then $Ker(\bar{f})$, $Coker(\bar{f})$ and $Im(\bar{f})$ all have even rank.
\end{lemma}
\begin{proof}
This is because a holomorphic map between K\"ahler manifolds is compatible with the Hodge structure. In fact, suppose $\varphi:X\to Y$ is such a holomorphic map, then
$$\varphi^*:H^1(Y,\mathbb{C})\to H^1(X,\mathbb{C})$$
preserves the Hodge decomposition. And by Poincare duality, $\varphi^*:H^1(Y,\mathbb{C})\to H^1(X,\mathbb{C})$ is dual to $\varphi_*:H_1(X,\mathbb{C})\to H_1(Y,\mathbb{C})$. Therefore by the isomorphism in Lemma \ref{even}, the result follows.
\end{proof}
\begin{lemma}\label{serre}
Every finite group is projective, hence K\"ahler.
\end{lemma}
This is observed by Serre, a proof can also be found in \cite{Shafarevich} page 227. There is a similar result for K\"ahler homomorphisms:
\newtheorem{prop}[thm]{Proposition}
\begin{lemma}
Suppose $f:A\to B$ is a K\"ahler homomorphism between K\"ahler groups, and $C$ is a finite group. If $h:A\to B\times C$ satisfies $p_1\circ h=f$, where $p_1$ is the projection to first component, then $h$ is K\"ahler. In particular, any homomorphism between K\"ahler groups is K\"ahler, if the target group is finite.
\end{lemma}
\begin{proof}
Suppose $A$, $B$ and $C$ are the fundamental groups of compact K\"ahler manifolds $X$, $Y$ and $Z$, and $f$ is represented by a holomorphic map $\varphi: X\to Y$, i.e., $\varphi_*:\pi_1(X)\to\pi_1(Y)$ is isomorphic to $f$.

Let $\tilde{X}$, $\tilde{Y}$ and $\tilde{Z}$ be the universal covers of $X$, $Y$ and $Z$, then $A$, $B$ and $C$ act on $\tilde{X}$, $\tilde{Y}$ and $\tilde{Z}$ respectively. Let $h_i=p_i\circ h$, for $i=1,2$, then $h_1=f$. Denote by $G(h_2)=\{(x,h_2(x))\in A\times C|x\in A\}$ the graph of $h_2$. It is a subgroup of $A\times C$ with finite index, so we can take the quotient of $\tilde{X}\times\tilde{Z}$ by this subgroup, which has a further quotient:
$$\tilde{X}\times\tilde{Z}/G(h_2)\stackrel{\pi}{\rightarrow} \tilde{X}\times\tilde{Z}/(A\times C)=X\times Z$$
Since $\pi$ is a finite covering map, $\tilde{X}\times\tilde{Z}/G(h_2)$ is K\"ahler and compact.
Composing $\pi$ with $(f,id)$, we get
$$\tilde{X}\times\tilde{Z}/G(h_2)\stackrel{\pi}{\rightarrow} X\times Z\stackrel{(f,id)}{\longrightarrow} Y\times Z$$
The corresponding homomorphisms of fundamental groups are
$$G(h_2)\to A\times C\to B\times C$$
where $G(h_2)$ is isomorphic to $A$, and the composition is isomorphic to $h:A\to B\times C$. Hence $h$ is K\"ahler.\\
\end{proof}
Notice that a finitely generated abelian group is K\"ahler if and only if it is projective, if and only if it has even rank. In fact, given any finitely generated abelian group of even rank, it is a direct sum of free part and torsion part. Therefore, according to Lemma \ref{serre}, we can take the fiber product of an abelian variety and a smooth projective variety. And the fundamental group of the product is isomorphic to the given abelian group. We will prove that, for a homomorphism between abelian groups, the obstruction of being K\"ahler is the same.
\begin{prop}
Let $f:A\to B$ be a homomorphism of abelian K\"ahler groups, then the following three conditions are equivalent:

$(i)$$f$ has even rank, i.e., one of $Ker(f), Coker(f), Im(f)$ has even rank, hence all of them.

$(ii)$$f$ is K\"ahler.

$(iii)$$f$ is projective.
\end{prop}

\begin{proof}
$(iii)\Rightarrow(ii)$ is trivial. $(ii)\Rightarrow(i)$ follows from Lemma 2.2.

For $(i)\Rightarrow (iii)$, suppose $f:A\to B$ is a homomorphism of abelian groups satisfying properties of $(i)$.

We first prove the case that $A$ and $B$ are both free of rank $2r$ and $2s$. Then $f$ is represented by a matrix, which we can assume to be a Smith normal form by a proper choice of basis. Hence $f$ factors into three steps, a projection to a direct summand, an injection with finite cokernel, an injection as a direct summand. Therefore, we will define a homomorphism $\phi:X\to Y$ between some abelian varieties as a composition of three steps, a projection to a direct summand, a finite covering map, and an injection as a direct summand. And the induced map of the fundamental groups $\phi_{*}:\pi(X)\to \pi(Y)$ will be isomorphic to $f$. 

Suppose the Smith normal form of $f$ is represented by $2r\times 2s$ matrix
$$\begin{pmatrix}
a_1&\ldots&0&\ldots&0\\
\vdots&\ddots&\vdots&&\vdots\\
0&\ldots&a_{2l}&\ldots&0\\
\vdots&&\vdots&&\vdots\\
0&\ldots&0&\ldots&0
\end{pmatrix}$$
We choose $X$ to be the $r$-dimensional torus $$\mathbb{C}^r/(a_1\mathbb{Z}+a_2i\mathbb{Z})+\ldots+(a_{2l-1}\mathbb{Z}+a_{2l}i\mathbb{Z})+(\mathbb{Z}+i\mathbb{Z})^{r-l}$$
The first step is projection to the torus of the first $l$ components\\
$$\mathbb{C}^l/(a_1\mathbb{Z}+a_2i\mathbb{Z})+\ldots+(a_{2l-1}\mathbb{Z}+a_{2l}i\mathbb{Z}))$$
In the second step, we take the covering map onto torus $\mathbb{C}^l/(\mathbb{Z}+i\mathbb{Z})^l$, as a further quotient. Then we take the embedding to $\mathbb{C}^s/(\mathbb{Z}+i\mathbb{Z})^s$, which will be our $Y$. All the tori we chose are abelian varieties, since they are products of elliptic curves.

The case that only $B$ is free follows from the previous case and Lemma \ref{serre}, since a homomorphism from a finite group to a free abelian group must be zero, and a finitely generated abelian group is a direct sum of it is free part and torsion part.

The general case follows from the previous case, and the previous lemma. 
\end{proof}
\section{Non-projective example}
In this section, we construct a K\"ahler homomorphism which is not projective.

In \cite{Voisin}, Voisin proved the following:
\begin{lemma}
Suppose an $n(\geq 2)$-dimensional complex torus $T$ has an endomorphism $\phi$, then $\phi_*$ acts on $H_1(T,\mathbb{Z})$. Let $p$ be the characteristic polynomial of $\phi_*$. If $p$ is irreducible, has no real roots and the Galois group of its splitting field acts as the symmetric group on the roots, then the torus $T$ is not an abelian variety.
\end{lemma}
She also gave the method to construct such a torus. A concrete example for $n=2$ can be found in \cite{Hu}.

Suppose $(T, \phi)$ satisfies the assumption of the Lemma, and $W$ is a compact K\"ahler manifold with $\pi_1(W)$ isomorphic to the symmetric group $S_3$. Denote the universal cover of $W$ by $\tilde{W}$. Let $S_3$ act on $(T\times T\times T)\times \tilde{W}$ by permutation on the first part and by deck transformation on the second. Since the action is fixed point free, we get a K\"ahler manifold as the quotient $(T\times T\times T)\times \tilde{W}/S_3$. Fixing any point $P$ on $\tilde{W}$, we define a sequence of maps:
$$T\times T\stackrel{\mu'}{\rightarrow}T\times T\times T\stackrel{i}{\rightarrow}T\times T\times T\times \tilde{W}\stackrel{\pi}{\rightarrow}T\times T\times T\times \tilde{W}/S_3$$
where we let the first map $\mu':(a,b)\mapsto (a+b, \phi(a)+b,b)$,  and $i$ is mapping to the fiber over $P\in \tilde{W}$, and $\pi$ is the quotient map.

Let $\mu$ be the composition $\pi\circ i\circ \mu'$, then the induced map $\mu_*$ of fundamental groups is K\"ahler. But it cannot be projective.

\begin{prop}
$\mu_*$ is not projective.
\end{prop}
\begin{proof}
Suppose $\nu:X\to Y$ is a holomorphic map between compact K\"ahler manifolds, which induces the same homomorphism of fundamental groups as $\mu$. We will show that $Y$ cannot be projective.\\
Since 
$$\pi_1(Y)\cong\pi_1((T\times T\times T)\times\tilde{W}/S_3)\cong(\mathbb{Z}^{2n}\times\mathbb{Z}^{2n}\times\mathbb{Z}^{2n})\rtimes S_3$$
we can take $\tilde{Y}$ as the 6-fold covering of $Y$ corresponding to the surjection $\mathbb{Z}^{2n}\times\mathbb{Z}^{2n}\times\mathbb{Z}^{2n}\rtimes S_3\to S_3$. $S_3$ will act on $\tilde{Y}$. Notice the composition of $\mu_*$ and the surjection $\mathbb{Z}^{2n}\times\mathbb{Z}^{2n}\times\mathbb{Z}^{2n}\rtimes S_3\to S_3$ is trivial. Hence $\nu:X\to Y$ factors through $X\stackrel{\nu'}{\rightarrow}\tilde{Y}\to Y$. The choice of $\nu'$ is not unique, but we will specify which one we want later.

Consider the induced map of Albanese: $Alb(\nu'):Alb(X)\to Alb(\tilde{Y})$, where $Alb(\tilde{Y})$ has an $S_3$ action on it induced by the action on $\tilde{Y}$. Let $\Sigma$ be the fixed locus of $Alb(\tilde{Y})$ by this $S_3$ action. Then $\Sigma$ is an n-dimensional subtorus. Also let $\Gamma_1$, $\Gamma_2$ and $\Gamma_3$ be the fixed loci by the three order 2 elements $s_1=(23), s_2=(31), s_3=(12)$ of $S_3$ respectively. Therefore $\Gamma_i$ corresponds to the permutation $s_i\in S_3$ fixing the $i$-th $T$. Therefore, $s_3$ gives an isomorphism between $\Gamma_1$ and $\Gamma_2$. Let $\Gamma_4$ be the image of $Alb(\nu')$.

Since as maps of the fundamental groups $\mu_*$ is isomorphic to $\nu_*$, lifting to the 6-fold covering, $\mu'_*\cong (i\circ \mu')_*$ must be isomorphic to $\nu'_*$. Therefore as real Lie groups, we can identify $T\times T\times T$ with $Alb(\tilde{Y})$, and hence parametrize $Alb(\tilde{Y})$ by elements in $T$. For instance, $\Sigma=\{(b,b,b)\}$, $\Gamma_1=\{(a+b,b,b)\}$, $\Gamma_2=\{(b,a+b,b)\}$. We choose $\nu'$ in the way such that according to the above identification, $Alb(\nu')$ is same as $Alb(\mu')$. Then $\Gamma_4=\{(a+b,\phi(a)+b,b)\}$. Therefore we can take quotients of those subtori of $Alb(\tilde{Y})$ by $\Sigma$. Denote the quotients with a bar on top. Now we get n-dimensional subtori $\overline{\Gamma_i}$ in the 2n-dimensional torus $\overline{Alb}\stackrel{\mathrm{def}}{=}\overline{Alb(\tilde{Y})}$. Obviously, $\overline{\Gamma_1}\cap\overline{\Gamma_2}=0$, we have a decomposition $\overline{Alb}=\overline{\Gamma_1}\oplus\overline{\Gamma_2}$. Projection to each component restricting to $\overline{\Gamma_4}$ gives two homomorphisms of tori, $\overline{\Gamma_4}\stackrel{pr_i}{\rightarrow}\overline{\Gamma_i}$, $i=1,2$, and $pr_1$ is an isomorphism.

Since $s_3$ permutes $\Gamma_1$ and $\Gamma_2$, and fixes $\Sigma$, it induces an isomorphism between $\overline{\Gamma_1}$ and $\overline{\Gamma_2}$, which we also call $s_3$ by abusing of notation. Now, consider the composition:
$$\overline{\Gamma_1}\stackrel{pr_1^{-1}}{\longrightarrow}\overline{\Gamma_4}\stackrel{pr_2}{\longrightarrow}\overline{\Gamma_2}\stackrel{s_3}{\longrightarrow}\overline{\Gamma_1}$$
It is same as $\phi$ via identifying $\overline{\Gamma_1}$ with $T$ as real tori. Here $\overline{\Gamma_4}$ plays the role of the graph. 

By the previous lemma, $\overline{\Gamma_1}$ is not an abelian variety. Consequently, none of $\Gamma_1$, $Alb(\tilde{Y})$, $\tilde{Y}$ or $Y$ can be a projective variety. Therefore $\mu_*$ is not a projective homomorphism.
\end{proof}

\end{document}